\documentclass[a4paper,draft]{amsproc}
\usepackage{amssymb}
\usepackage[hyphens]{url} 
\theoremstyle{plain}
\newtheorem*{theoA}{Theorem A}
\newtheorem*{theoB}{Theorem B}
\newtheorem*{theoC}{Theorem C}
\newtheorem*{theoD}{Theorem D}
\newtheorem*{theoE}{Theorem E}
\newtheorem*{theoF}{Theorem F}
\newtheorem{theo}{Theorem}[section]
\newtheorem{lem}{Lemma}[section]
\newtheorem{cor}{Corollary}[section]
\newtheorem{exm}{Example}[section]
\newtheorem{defi}{Definition}[section]
\newtheorem{ques}{Question}[section]
\theoremstyle{definition}
\theoremstyle{remark}
\newtheorem{rem}{Remark}[section]

\newcommand{\ol}{\overline}
\newcommand{\be}{\begin{equation}}
	\newcommand{\ee}{\end{equation}}
\newcommand{\beas}{\begin{eqnarray*}}
	\newcommand{\eeas}{\end{eqnarray*}}
\newcommand{\bea}{\begin{eqnarray}}
	\newcommand{\eea}{\end{eqnarray}}
\renewcommand{\d}{\displaystyle}
\newtheorem*{ack}{Acknowledgement}
\numberwithin{equation}{section}

\renewcommand{\leq}{\leqslant}\renewcommand{\geq}{\geqslant}

\title[A GENERAL CHARACTERIZATION ON THE UNIQUENESS PROBLEM...]{A GENERAL CHARACTERIZATION ON THE UNIQUENESS PROBLEM OF L-FUNCTIONS AND GENERAL MEROMORPHIC FUNCTIONS}



\author[Mallick]{\bfseries Sanjay Mallick}
\address{
	Department of Mathematics \\ 
	Cooch Behar Panchanan Barma University   \\ 
	Cooch Behar\\
	India}
\email{sanjay.mallick1986@gmail.com, smallick.ku@gmail.com}

\author[Saha]{\bfseries Ripan Saha}
\address{ 	Department of Mathematics \\ 
	Cooch Behar Panchanan Barma University   \\ 
	Cooch Behar\\
	India }
\email{ripan.tfg.saha@gmail.com}

\thanks{Supported by SERB, DST, INDIA and NBHM, DAE, INDIA} 
\thanks{Communicated by ...}

\begin{document}
	
	{\begin{flushleft}\baselineskip9pt\scriptsize
			MANUSCRIPT
	\end{flushleft}}
	\vspace{18mm} \setcounter{page}{1} \thispagestyle{empty}

	\begin{abstract}
		 In the paper, concerning a question of Yi \cite{yuan2018}, we study general criterion for the uniqueness of an L-function and a general meromorphic function. Our results improve and extend all the existing results in this direction \cite{yuan2018,ss,sh,bk} to the most general setting. Moreover, we have exhibited a handsome number of examples to justify our claims as well as to confirm the wide-ranging applications of our results.  
	\end{abstract}
	
	\maketitle
	
	\section{Introduction and Related Results} 
	
	 The value distribution  of L-functions in the Selberg class via Nevanlinna theory is a recent trend to study different properties of L-functions as well as uniquely determining the same. This kind of L-functions includes the Riemann zeta function $\zeta(s) = \sum\limits_{n=1}^{\infty} n^{-s}$ and essentially those Dirichelet series where one might expect a Riemann hypothesis. Such an L-function is
	defined \cite{selberg1992,steuding2007value} to be a Dirichelet series
	\be\nonumber
	\mathcal{L}(s) = \sum\limits_{n=1}^{\infty} \frac{a(n)}{ n^{-s}}
	\ee
	with $a(1)=1$ satisfying the following axioms:
	\begin{itemize}
		\item(i) \emph{Ramanujan hypothesis :} $a(n)\ll n^{\varepsilon}$ for every $\varepsilon> 0;$
		\item(ii) \emph{Analytic continuation :} There is a non-negative integer $m$ such that\\ $(s-1)^m\mathcal{L}(s)$ is an entire function of
		finite order;
		\item (iii) \emph{Functional equation:} $\mathcal{L}$ satisfies a functional equation of type \begin{equation*}
			\Lambda_\mathcal{L}(s) = \omega\ol{ \Lambda_\mathcal{L}(1-\ol{s})},
		\end{equation*}
		where \begin{equation*}
			\Lambda_\mathcal{L}(s) = \mathcal{L}(s)Q^s\prod_{j=1}^{k}\Gamma(\lambda_js+\nu_j),
		\end{equation*}
		with positive real numbers $Q, \lambda_j$ and complex numbers $ \nu_j, \omega$ with Re$\nu_j \geq 0$ and $|\omega| = 1;$
		\item (iv) \emph{Euler product hypothesis :} $\log \mathcal{L}(s) =  \sum\limits_{n=1}^{\infty} \frac{b(n)}{ n^{s}}$,  where $b(n) = 0$ unless $n$ is a positive power of a prime
		and $b(n) \ll n^{\theta}$ for some $\theta < \frac{1}{2}$.
	\end{itemize}
	\par Throughout the paper by an L-function we always mean an L-function of the above kind and by any meromorphic function we always mean a meromorphic function defined in $\mathbb{C}$. We denote $\mathbb{\ol C}=\mathbb{C}\cup\{\infty\}$. By $\mathbb{N}$ we  mean the set of all natural numbers.  Though for standard definitions used in this paper we refer our readers to follow \cite{wkh}, yet for the sake of our convenience we denote the order of $f$ by $\rho(f)$, where 
	\be\nonumber\rho(f)=\limsup_ {r\rightarrow\infty} \frac{\log(T(r,f))}{\log r}.\ee
	By $S(r, f)$ we mean any quantity satisfying $S(r, f) = O(\log(rT(r, f)))$ for all $r$ possibly outside a set of finite linear measure. If $f$ is a function of finite order, then $S(r, f) = O(\log r)$
	for all $r$. 
	\vspace{0.1in}\par However, in recent times mainly during the last ten years,  a lot of investigations \cite{kac99,li2010,yuan2018,hu2016,selberg1992,steuding2007value}  have been pursued by various authors in the direction of value distribution of L-functions. In due course of time, this research has been moved towards uniquely determining the L-functions with respect to the meromorphic functions having finitely many poles so that we may know about various  properties of L-functions via the established literature of this special kind of meromorphic functions. Below we recall some of these results and their gradual developments. Before that, we need to address the following basic definitions to make out the results clearly. 
	\begin{defi}\cite{3a}
		For a  non-constant meromorphic function $f$ and  $a\in{\mathbb{C}}$, let  $E_{f}(a)=\{(z,p)\in\mathbb{C}\times\mathbb{N}: f(z)=a\; with\; multiplicity\; p\}$ \\$ \left(\ol  E_{f}(a)=\{(z,1)\in\mathbb{C}\times\mathbb{N}: f(z)=a\}\right)$, then we say  $f$, $g$ share the value $a$ CM(IM) if $E_{f}(a)=E_{g}(a)$$\left( \ol E_{f}(a)=\ol E_{g}(a)\right).$ For $a=\infty$, we define $E_f(\infty) := E_{1/f}(0)$ $\left( \ol E_f(\infty) :=\ol  E_{1/f}(0)\right)$.
	\end{defi}
	\begin{defi}\cite{3a}
		For a  non-constant meromorphic function $f$ and  $S\subset\overline{\mathbb{C}}$, let $E_{f}(S)=\bigcup_{a\in S}\{(z,p)\in\mathbb{C}\times\mathbb{N}: f(z)=a\; with\; multiplicity\; p\}$\\ $\left(\ol  E_{f}(S)=\bigcup_{a\in S}\{(z,1)\in\mathbb{C}\times\mathbb{N}: f(z)=a\}\right) $, then we say  $f$, $g$ share the set $S$ CM\\(IM) if $E_{f}(S)=E_{g}(S)$ $\left(\ol  E_{f}(S)=\ol E_{g}(S)\right) $. 
	\end{defi}
	\begin{defi}\cite{lahiriNagoya, Lahiri CVEE}
		Let k be a non-negative integer or infinity. For $a\in\overline{\mathbb{C}}$  we
		denote by $E_k(a; f)$ the set of all a-points of f, where an $a$-point of multiplicity $m$ is counted
		$m$ times if $m\leq k$ and $k + 1$ times if $m > k$. If $E_k(a; f) = E_k(a; g)$, we say that $f$, $g$ share
		the value a with weight $k$.
	\end{defi}
	We write $f$, $g$ share $(a,k)$ to mean that $f$, $g$ share the value a with weight $k$. Clearly if
	$f$, $g$ share $(a,k)$ then $f$, $g$ share $(a,p)$ for any integer $p$, $0\leq p < k$. Also we note that $f$, $g$
	share a value $a$ IM or CM if and only if $f$, $g$ share $(a,0)$ or $(a,\infty)$ respectively.
	\begin{defi}\cite{lahiriNagoya}
		For $S\subset\mathbb{\ol C}$ we define $E_f(S,k)=\cup_{a\in S}E_k(a;f)$, where k is a non-negative integer $a\in S$ or infinity. Clearly $E_f(S)=E_f(S,\infty)$ and $\overline{E}_f(S)=E_f(S,0)$. If $E_f(S,k)=E_g(S,k)$, then we say that $f$ and $g$ share the set $S$ with weight $k$.
	\end{defi}
	Obviously Definition 1.3 and Definition 1.4 are the refined notions of Definition 1.1 and Definition 1.2 respectively. However, now we go for the results as mentioned above. 
	\begin{theoA}\cite{li2010}
		Let a and b be two distinct
		finite values, and let f be a meromorphic function in
		the complex plane such that f has finitely many poles
		in the complex plane. If f and a non-constant L-function
		$\mathcal{L}$ share $(a,\infty)$ and $(b,0)$, then $\mathcal{L} = f$.
	\end{theoA}
In 2018, taking the famous Gross Problem \cite{fg} into account,
Yuan, Li and Yi \cite{yuan2018} generalized the uniqueness problem of $\mathcal{L}$ and $f$ under the aegis of shared sets and   proposed the following inevitable question.
\begin{ques}\label{q1}\cite{yuan2018}
What can be said about the relationship between a meromorphic function $f$ and an L-function $\mathcal{L}$ if $f$ and $\mathcal{L}$ share one or two sets?
\end{ques}
Apropos of {\em Question \ref{q1}}, Yuan, Li and Yi provided the following result.
\begin{theoB}\cite{yuan2018}
Let $Q(z) = z^n + az^m + b$, where $a$, $b$ are non-zero constants with $\gcd(m,n)=1$ and $n \geq 2m + 5$. Further suppose  $f$ be a non-constant meromorphic function having finitely many poles and $\mathcal{L}$
be a non-constant L-function such that $E_{f}(S,\infty)=E_{\mathcal{L}}(S,\infty)$, where $S=\{z:Q(z)=0\}$. Then $f = \mathcal{L}$.
\end{theoB}
Later on,  Sahoo-Sarkar \cite{ss} improved  {\em Theorem B} by relaxing the weight  of the shared set from CM to weight 2.
\begin{theoC}\cite{ss}
Let $S$ be defined same as in {\em Theorem B} and $n \geq 2m+5$. Suppose $f$ be a non-constant meromorphic
function having finitely many poles in $\mathbb{C}$ and  $\mathcal{L}$ be a non-constant L-function. If $f$ and $\mathcal{L}$
share $(S, 2)$, then $f = \mathcal{L}$.
\end{theoC}
At the same time, Sahoo-Halder also considered the same problem with weight `$0$' and proved the following result. 
\begin{theoD}\cite{sh}
Let $S$ be defined same as in {\em Theorem B} and $n \geq \max\{2m+5,4q+9\},$ where $q=n-m\geq1$. Let $f$ be a non-constant meromorphic
function having finitely many poles in $\mathbb{C}$ and  $\mathcal{L}$ be a non-constant L-function. If $f$ and $\mathcal{L}$
share $(S, 0)$, then $f = \mathcal{L}$.
\end{theoD}
Though  {\em Theorem C} and {\em Theorem D} were later proved to be wrong by Banerjee-Kundu in \cite{bk}. Hence Banerjee-Kundu provided the following result rectifying the gaps of  these theorems.
\begin{theoE}\cite{bk}
Let $S$ be defined as in {\em Theorem B}, f be a non-constant meromorphic function having finitely
many poles in $\mathbb{C}$ and $\mathcal{L}$ be a non-constant L-function such that $E_f(S, t) = E_\mathcal{L}(S, t)$. If
\begin{enumerate}
	\item[(i)] $t \geq 2$ and $n \geq 2m + 5$, or
	\item[(ii)] $t = 1$ and $n \geq 2m + 6$, or
	\item[(iii)] $t=0$ and $n\geq 2m+11$, \\then $ f = \mathcal{L}$.
\end{enumerate} 	
\end{theoE}
In the same paper, Banerjee-Kundu also proved the following result analogous to {\em Theorem E}.
\begin{theoF}\cite{bk}
Let $S=\{z:z^{n}+az^{n-m}+b=0\}$, where $a,b$ are non-zero constants and $\gcd(n,m)=1$.  Let $f$ be a non-constant meromorphic function having finitely
many poles in $\mathbb{C}$ and $\mathcal{L}$ be a non-constant L-function such that $E_f(S, t) = E_\mathcal{L}(S, t)$. If
\begin{enumerate}
	\item[(i)] $t \geq 2$ and $n \geq 2m + 5$, or
	\item[(ii)] $t = 1$ and $n \geq 2m + 6$, or
	\item[(iii)] $t=0$ and $n\geq 2m+11$, \\then $ f = \mathcal{L}$.
\end{enumerate}   	\end{theoF}
Now from {\em Theorem B-F} one would naturally have the following observations.
\begin{enumerate}
\item[(i)] All the authors always used one of the following polynomials. \be\label{e000} P(z)=z^{n}+az^{m}+b\;\; or\;\; P(z)=z^{n}+az^{n-m}+b,\ee where $a,b$ are non-zero constants and $\gcd(n,m)=1$.  
\item[(ii)] The authors always considered a special class  of meromorphic functions; i.e., meromorphic functions having finitely many poles for the uniqueness of the same with an L-function.
\item[(iii)] The authors also considered that the set of zeros of the polynomials given by (\ref{e000}) may have multiple zeros. 
\end{enumerate}
\vspace{0.1in}\par  
Apropos of observation (i), One would naturally raise the following question.
\begin{ques}
Are these the only polynomials  given by (\ref{e000}) whose set of zeros provide uniqueness of $f$ and $\mathcal{L}$?
\end{ques} 
If the answer is no. Then we would be happy to have the answer of the following general question.
\begin{ques}
Can we have a general characterization of polynomials  whose set of zeros provide uniqueness of $f$ and $\mathcal{L}$?
\end{ques}  
With respect to observation $(ii)$, we would like to recall that {\em Question \ref{q1}} was made for general meromorphic functions, whereas all the theorems (Theorem B-F) under the same question were treated only for meromorphic functions having finitely many poles. In this situation,  one would naturally have the following query.
\begin{ques}\label{q12}
Is it possible to extend observation $(ii)$ for the whole class of meromorphic functions in order to have a full answer of {\em Question \ref{q1}}?
\end{ques}

Finally with respect to observation (iii), we have the following remark.
\begin{rem}
	Very recently	in \cite[see paragraph between Theorem H and Theorem I]{bk1} Banerjee-Kundu have shed a light on  the fact that all the results 
	from {\em Theorem B-F} are valid only when the polynomials  have simple zeros or the polynomials have multiple zeros with weight of the shared set as $0$ and not valid otherwise. Moreover, a close look in the proof of these theorems reveals that the concept of multiple zeros of polynomials does not provide us with any such analytical or technical advantage. 
\end{rem}

Hence in this paper, we  solely concentrate on the polynomials having only simple zeros and answer all the above questions from {\em Question \ref{q1}-\ref{q12}} affirmatively. 
In fact, we present   general criterions for any general polynomial so that the set of zeros of the same would provide the uniqueness of  $\mathcal{L}$ with a general meromorphic function $f$ instead of a meromorphic function having finitely many poles. That is, we extend all these results to the most general setting. As a consequence, all the above results comes under the special cases of our main theorems. 
Moreover, our main theorems significantly improve all the above results obtained for ignoring multiplicities providing better estimates of least cardinalities; i.e., $11$ instead of $13$.   In a nutshell, our main theorems bring all the existing theorems under a single umbrella in a more improved version with extent to the most general setting. 

\vspace{0.1in}\par In the {\bf ``Application"} section we have  exhibited a number of examples proving all our claims to be true and showing the far reaching applications of our  results.
\par Before going to our main results, we make a short discussion on the structure of a general
polynomial as this will play a crucial role throughout the paper. 
\vspace{0.1in}\par Let us consider the following general polynomial $P(z)$ of degree $n$ having only simple zeros.  	\bea\label{el0} P(z)=a_nz^n+a_{n-1}z^{n-1}+\ldots+a_1z+a_0,\eea
where $ a_0,a_1,\dots,a_n$ are  complex numbers with $a_n,a_0\neq 0$, $a_{i}$ being the first non-vanishing coefficient  from $a_{n-1}, a_{n-2},\ldots,a_{1}$. Let \be\label{el00}S=\{z:P(z)=0\}.\ee
Observe that (\ref{el0}) can be written in the form \be\label{el000}P(z)=a_{n}\prod\limits_{i=1}^{p}(z-\alpha_{i})^{m_{i}}+a_{0},\ee where $p$ denotes the   number of distinct zeros of $P(z)-a_{0}$. Let us also denote by $s$ the number of distinct zeros of $P^{'}(z)$. Hence we would have 
\be\label{el0000}P^{'}(z)=na_{n}\prod\limits_{i=1}^{s}(z-\eta_{i})^{r_{i}},\ee where $r_{i}$ denotes the multiplicities of distinct zeros of $P^{'}(z)$.
\vspace{0.2in}\par 	
Set 
\bea\label{el2} R(z)&=& -\frac{a_nz^n}{a_iz^i+a_{i-1}z^{i-1}+\ldots+a_1z+a_0}\\\nonumber&=&- \frac{a_{n}z^n}{a_{i}\prod\limits_{j=1}^{k}(z-\beta_j)^{m_{j}}}\\\nonumber&=&\d-\frac{a_{n}z^{n}}{\phi(z)},\eea
where $a_{0},a_{1},\ldots,a_{n}, a_{i}$ are as defined in (\ref{el0}) and $\beta_1,\beta_2,\ldots,\beta_k$ are the roots of the equation $$\phi(z)=a_iz^i+a_{i-1}z^{i-1}+\ldots+a_1z+a_0=0,$$ with multiplicities $m_1,m_2,\ldots,m_k$. Clearly \be\nonumber R(z)-1=\d-\frac{P(z)}{\phi(z)},\ee where $P(z)$ is defined by (\ref{el0}) and obviously $P(z)$ and $\phi(z)$ do not share any common zero. Hence $S$ as defined in (\ref{el00}) can be treated as \be\label{el02}S=\{z:P(z)=0\}=\{z:R(z)-1=0\}.\ee  Let $R^{'}(z)$ has $l$ distinct zeros say $\delta_{1},\delta_{2},\ldots,\ldots,
\delta_{l}$ with multiplicities $q_{1},q_{2},\ldots,q_{l}$ respectively. Then  
From (\ref{el2}) we would have 
\bea \label{el3} R'(z)&=&\frac{\gamma\prod\limits_{j=1}^{l}(z-\delta_j)^{q_{j}}}{\prod\limits_{j=1}^{k}(z-\beta_j)^{p_{j}}},\eea where $\gamma\in\mathbb{C}-\{0\}$ and $p_{j}\in\mathbb{N}$ for all $j\in\{1,2,\ldots,k\}$. 
\begin{rem}
	Observe that in the definition (\ref{el0}) of the general polynomial $P(z)$, the condition $a_{i}\neq 0$ for at least one $i\in \{1,2,\ldots,n-1\}$ is necessary. Because otherwise we would find a non-constant L-function $\mathcal{L}$ and a non-constant meromorphic function $f$ which share the set $S=\{z:P(z)=0\}$ CM but $f\neq \mathcal{L}$.
	\par  For example, let  $a_{i}=0$ for all $i\in \{1,2,\ldots,n-1\}$. Then $S=\{z: a_{n}z^{n}+a_{0}=0\}$. Consider a non-constant L-function $\mathcal{L}$ and  a non-constant meromorphic function $f$ such that $f=\zeta\mathcal{L}$, where $\zeta$ is the nth root of unity such that $\zeta \ne 1 $. Then clearly, $$a_{n}f^{n}+a_{0}=a_{n}\mathcal{L}^{n}+a_{0};$$ $$i.e.,\;\; \prod\limits_{i=1}^{n}(f-\sigma_{i})=\prod\limits_{i=1}^{n}(\mathcal{L}-\sigma_{i}),$$ where $\sigma_{i}\in S$ for $i=\{1,2,\ldots,n\}$. That is $f$ and $\mathcal{L}$ share $S$ CM but $f\ne \mathcal{L}$.
\end{rem}  

\section{Main Results}   	
The following two theorems are the main results of this paper.
\begin{theo}\label{t1}
	Let $P(z)$ be given by (\ref{el000}) with $p\geq2$, and $S$,$s$ be defined by (\ref{el00}),(\ref{el0000}) respectively.
	Suppose $f$ be a non-constant meromorphic function and  $\mathcal{L}$
	be a non-constant L-function sharing $(S,t)$.
	Then for $n \geq \max \left\{2p+2,2s+5\right\}$, when $t\geq2$; $n \geq \max \left\{2p+2,2s+6\right\}$, when $t=1$ and for\\ $n \geq \max \left\{2p+2,2s+11\right\}$, when $t=0$; the following are equivalent :
	\begin{itemize}
		\item[(i)] $P(f)=P(\mathcal{L}) \implies f=\mathcal{L}$ ;
		\item[(ii)] $E_f(S, t)=E_\mathcal{L}(S,t)\implies f=\mathcal{L}$.
	\end{itemize}
\end{theo}
\begin{theo}\label{t2}
	Let $R(z)$ be defined by (\ref{el2}) with $k\in \mathbb{N}$ and $S, l$ be defined by (\ref{el02}), (\ref{el3}) respectively with \begin{enumerate}
		\item[(a)] $k\geq2$, or
		\item[(b)] $k=1$ and  $n\geq 3b_1+1$, where $b_1=gcd(m_1,n)$.
	\end{enumerate} 
	Let $f$ be a non-constant meromorphic function and  $\mathcal{L}$
	be a non-constant L-function sharing $(S,t)$.
	Then for $n \geq \max \left\{2k+4,2l+5\right\}$, when $t\geq2$;\\ $n \geq \max \left\{2k+4,2l+6\right\}$, when $t=1$ and for $n \geq \max \left\{2k+4,2l+11\right\}$, when $t=0$; the following are equivalent :
	\begin{itemize}
		\item[(i)] $R(f)=R(\mathcal{L}) \implies f=\mathcal{L}$ ;
		\item[(ii)] $E_f(S, t)=E_\mathcal{L}(S,t)\implies f=\mathcal{L}$.	\end{itemize}	
	
\end{theo} 
Next we state the following two corollaries with respect to {\em Theorem \ref{t1}} and {\em Theorem \ref{t2}} respectively for meromorphic functions having finitely many poles. 
\begin{cor}\label{c1}
	Let $P(z)$, $p$, $S$, $s$ be as defined in {\em Theorem \ref{t1}}.
	Suppose $f$ be a non-constant meromorphic function having finitely many poles and  $\mathcal{L}$
	be a non-constant L-function sharing $(S,t)$.
	Then for $n \geq \max \left\{2p+1,2s+3\right\}$, when $t\geq2$; $n \geq \max \left\{2p+1,2s+4\right\}$, when $t=1$ and for $n \geq \max \left\{2p+1,2s+7\right\}$, when $t=0$; the following are equivalent :
	\begin{itemize}
		\item[(i)] $P(f)=P(\mathcal{L}) \implies f=\mathcal{L}$ ;
		\item[(ii)] $E_f(S, t)=E_\mathcal{L}(S,t)\implies f=\mathcal{L}$.
	\end{itemize}
\end{cor}
\begin{cor}\label{c2}
	Let $R(z)$, $k$,  $S, l$ be as defined in {\em Theorem \ref{t2}} with \begin{enumerate}
		\item[(a)] $k\geq2$, or
		\item[(b)] $k=1$ and  $n\geq 2b_1+1$, where $b_1=gcd(m_1,n)$.
	\end{enumerate} 
	Let $f$ be a non-constant meromorphic function having finitely many poles and  $\mathcal{L}$
	be a non-constant L-function sharing $(S,t)$.
	Then for $n \geq \max \left\{2k+3,2l+3\right\}$, when $t\geq2$; $n \geq \max \left\{2k+3,2l+4\right\}$, when $t=1$ and for $n \geq \max \left\{2k+3,2l+7\right\}$, when $t=0$; the following are equivalent :
	\begin{itemize}
		\item[(i)] $R(f)=R(\mathcal{L}) \implies f=\mathcal{L}$ ;
		\item[(ii)] $E_f(S, t)=E_\mathcal{L}(S,t)\implies f=\mathcal{L}$.	\end{itemize}	
	
\end{cor}   
\section{Applications}
In this section, we exhibit different examples applying our main theorems and corollaries which clearly shows the broad perspectives of our results with special attention to the fact that the results  {\em Theorem B-F} comes under the special cases of our results. In addition, our results significantly improve {\em Theorem B-F} in case of weight `$0$' of the shared set. Below we explain this fact via the following two examples.
\begin{exm}\label{ex2}
	Consider the polynomial 
	\be\nonumber
	P(z)=z^n+az^{n-m}+b= z^{n-m}(z^{m}+a)+b,
	\ee
	where $n, m$  are relatively prime inegers with $n-m>1$and a, b are non-zero constants such that the polynomial has no multiple zero. Suppose $S=\{z:P(z)=0\}$.
	Here $$p=m+1\geq 2 \;\ \text{and }\;\ gcd(n,n-m)=1.$$Also \be\nonumber
	P'(z)=z^{n-m-1}(nz^{m}+a(n-m));
	\ee
	i.e., $$s=m+1.$$

	Suppose that $P(f)=P(\mathcal{L})$, for any non-constant meromorphic function f and a non-constant L-function $\mathcal{L}$,  then we have
	\be\label{411}
	f^n-\mathcal{L}^n=-a(f^{n-m}-\mathcal{L}^{n-m}).
	\ee
	If $f^n \not\equiv \mathcal{L}^n$, then we can rewrite (\ref{411}) as
	\be\label{412}
	\mathcal{L}^{m}=-a\frac{(h-v)(h-v^2)...(h-v^{{n-m}-1})}{(h-u)(h-u^2)...(h-u^{n-1})},
	\ee
	where $h =\displaystyle\frac{f}{\mathcal{L}}$, $u = exp(2\pi i/n)$ and $v = exp(2\pi i/(n-m))$. Noting that n and (n-m) are relatively prime positive
	integers, then the numerator and denominator of (\ref{412}) have no common factors. Since $\mathcal{L}$ has atmost one pole at $z = 1$ in
	the complex plane, and whenever $n\geq 5$ we can see that there exists at least three distinct roots of $h^n = 1$ such that they are Picard
	exceptional values of h$,$ and so it follows by (\ref{412}) that h and thus $\mathcal{L}$ are constants, which is impossible.
	
	Therefore, we must have $f^n = \mathcal{L}^n$. Then by (\ref{411}) we also have $f^{n-m} = \mathcal{L}^{n-m}$. Since n and (n-m) are relatively prime positive integers, we deduce that $f = \mathcal{L}$. Thus we see that $P(f)=P(\mathcal{L})\implies f=\mathcal{L}$, when $n\geq 5$. 
	
	Now we apply {\em\bf Theorem \ref{t1}} to find the minimum value of $n$ for which we can say that  $E_f(S, t)=E_\mathcal{L}(S,t)\implies f=\mathcal{L}$.
	
	Therefore,  $E_f(S, t)=E_\mathcal{L}(S,t)\implies f=\mathcal{L}$ for  
	\begin{enumerate}
		\item $n \geq \max\{ 2m+4,2m+7\}=2m+7$ when $t\geq 2 $,
		\item  $n \geq \max\{ 2m+4,2m+8\}=2m+8$ when $t=1 $, and
		\item  $n \geq \max\{ 2m+4,2m+13\}=2m+13$ when $t=0 $.
	\end{enumerate}
	
	Again if $f$ is a non-constant meromorphic function with finitely many poles and $\mathcal{L}$ is a non-constant $L$ function, then by {\em\bf Corollary \ref{c1}} we have $E_f(S, t)=E_\mathcal{L}(S,t)\implies f=\mathcal{L}$ for  
	\begin{enumerate}
		\item $n \geq \max\{ 2m+3,2m+5\}=2m+5$ when $t\geq 2 $,
		\item  $n \geq \max\{ 2m+3,2m+6\}=2m+6$ when $t=1 $, and
		\item  $n \geq \max\{ 2m+3,2m+9\}=2m+9$ when $t=0 $.
	\end{enumerate} 
\end{exm}
\begin{exm}\label{ex5}
	Consider the polynomial \bea\nonumber  P(z) = z^n + az^m + b,\eea where $m$ and $n$ are positive integers such that $n \geq m+4$, $a$ and $b$ are finite non-zero complex numbers with $\frac{b^{n-m}}{a^m}\ne \frac{(-1)^nm^m(n-m)^{n-m}}{n^n}$. Then $P(z)$ has only simple zeros. Let $S$ denotes the set of zeros of $P(z)$. Suppose \bea\label{eyi} 
	R(z)=- \frac{z^n}{az^m+b}.\eea
	Then we find that $S=\{z:R(z)-1=0\}$.
	From (\ref{eyi}) we have \bea \nonumber
	R'(z)=- \frac{z^{n-1}[a(n-m)z^m+bn]}{(az^m+b)^2}.\eea 
	Now for a non-constant meromorphic function $f$ and a non-constant L-function $\mathcal{L}$ consider $R(f)=R(\mathcal{L})$. Then we have  \bea\label{e49} 
	\frac{f^n}{af^m+b}&=&\frac{\mathcal{L}^n}{a\mathcal{L}^m+b}\nonumber\\
	\implies a(f^n\mathcal{L}^m-f^m\mathcal{L}^n)-b(\mathcal{L}^n-f^n)&=&0.\eea
	Let   $h=\frac{f}{\mathcal{L}}$. Suppose that $h $ is a non-constant meromorphic function. Then from (\ref{e49}) we have
	\bea \label{e111}&&
	ah^{m}\mathcal{L}^{n+m}(h^{n-m}-1)+b\mathcal{L}^n(h^n-1)=0\\\nonumber&&
	\implies \mathcal{L}^m=-\frac{b(h^n-1)}{ah^{m}(h^{n-m}-1)}=-\frac{b}{a}\frac{(h-u)(h-u^2)...(h-u^{n-1})}{h^m(h-v)(h-v^2)...(h-v^{n-m-1})},
	\eea
	where  $u = exp(2\pi i/n)$, and $v = exp(2\pi i/(n-m))$. Since $n$ and $m$ are co-prime, so is $n$ and $n-m$. Hence the numerator and denominator of (\ref{e111}) have no common factors. Further, the function $\mathcal{L}$ has atmost one pole  in
	the complex plane, it follows that $h$ has atleast $n-m-1$ picard exceptional values among $\{0,v,v^2,...,v^{n-m-1}\}$.
	
	Clearly this is a contradiction as $n\geq m+4$. Hence  $h$ is  constant. Thus from (\ref{e111}) we must have  $h^{n}=1=h^{n-m}$ which in turn implies $h=1$; i.e., $f=\mathcal{L}$. \vspace{0.1in}\par Now we count the cardinality of the set $S$ for which  $E_f(S, t)=E_\mathcal{L}(S,t)\implies f=\mathcal{L}$. In this case, for $R(z)$ we have  $$l=m+1,\;\; k=m.$$  Therefore we obtain that $R(z)$ satisfies condition (a)[when $m\geq2$], or (b)[when m=1] and (i)  of
	{\em\bf Theorem \ref{t2}}. Hence  $E_f(S, t)=E_\mathcal{L}(S,t)\implies f=\mathcal{L}$ for \begin{enumerate}
		\item $n\geq \max\{2m+4,2m+7\}=2m+7$ when $t\geq 2$ ,
		\item $n\geq \max\{2m+4,2m+8\}=2m+8$ when $t= 1$, and 
		\item $n\geq \max\{2m+4,2m+13\}=2m+13$ when $t=0$.
	\end{enumerate}
	Now let $f$ be a non-constant meromorphic function with finitely many poles and $\mathcal{L}$ be a non-constant $L$ function. Then by {\em\bf Corollary \ref{c2}} we have $E_f(S, t)=E_\mathcal{L}(S,t)\implies f=\mathcal{L}$ for  
	\begin{enumerate}
		\item $n \geq \max\{ 2m+3,2m+5\}=2m+5$ when $t\geq 2 $,
		\item  $n \geq \max\{ 2m+3,2m+6\}=2m+6$ when $t=1 $, and
		\item  $n \geq \max\{ 2m+3,2m+9\}=2m+9$ when $t=0 $.
	\end{enumerate} 		
\end{exm} 	
\begin{rem}
	Observe that the conclusions obtained in {\em Example \ref{ex2}} and {\em Example \ref{ex5}} for meromorphic functions having finitely many poles are at par with all the existing results from {\em Theorem B-F} with an improvement in the case of weight ` $0$' $($ i.e., the cardinality is reduced from $2m+11$ to $2m+9)$.
\end{rem}
Another point we would like to explore on this occasion that the polynomials used in {\em Theorem B-F} are all critically injective \cite[see Lemma 4]{bk}. For the sake of our readers, we recall the definition of this particular type of polynomials. 
\begin{defi}
	Let $P(z)$ be a polynomial such that $P^{\prime}(z)$ has mutually $r$ distinct  zeros given by $d_{1}, d_{2}, \ldots, d_{r}$ with multiplicities $q_{1}, q_{2}, \ldots, q_{r}$ respectively. Then $P(z)$ is said to be a critically injective polynomial if $P(d_i)\not =P(d_j)$ for $i\not=j$, where $i,j\in \{1,2,\cdot\cdot\cdot,r\}$.
	\par 	Any polynomial which is not critically injective is called a non-critically injective polynomial.
\end{defi}
Obviously, any polynomial is either critically injective or non-critically injective. So, taking this fact into account, we have framed our theorems in such a way that the non-critically injective polynomials can also be fit into our theorems.  Even, the polynomials which are still unknown to be critically injective or non-critically injective also fit into our theorems. That is our results count all kinds of polynomials answering {\em Question \ref{q1}-\ref{q12}} affirmatively. Below we provide examples of each these three types of polynomials respectively, justifying all the above claims.	
\begin{exm}
	Let us consider the  polynomial :
	\bea\label{e44.0} P(z)=z^n+az^{n-m}+bz^{n-2m}+c,\eea
	where $a, b, c\in\mathbb{C}^*$ be such that $P(z)$ has no multiple root, $gcd(m, n) = 1$ with $n\geq2m+4$ and $\frac{a^2}{4b}=\frac{n(n-2m)}{(n-m)^2}$, $c\ne \frac{\beta_i\beta_j}{\beta_i+\beta_j}$. Here $\beta_i=-(c_i^n+ac_i^{n-m}+bc_i^{n-2m}),$ where $c_i$ are the roots
	of the equation $nz^{2m} + (n- m)az^m + b(n-2m) = 0$, for $i = 1, 2, \ldots, 2m$. 
	Suppose $S$ denotes the set of zeros of (\ref{e44.0}).
	\vspace{0.1in}\par  Obviously,  $P(z)$ has only simple zeroes and it is critically injective \cite[see Lemma 2.7]{3a}. 
	From (\ref{e44.0}) we have \bea\label{e440} P'(z)&=&z^{n-2m-1}[nz^{2m}+a(n-m)z^m+b(n-2m)]\\\nonumber&&=nz^{n-2m-1}\left(z^m+\frac{a(n-m)}{2n}\right)^2.\eea
	From (\ref{e44.0}) and (\ref{e440}) we find that $$p=2m+1 \;\;\; \mbox{and} \;\;\;s=m+1.$$
	In \cite[see proof of Theorem 1.1]{3a}, it is also proved that $P(f)=P(g)$ implies $f=g$ for $n\geq 2m+4$, where $f$and $g$ are non-constant meromorphic functions. Hence for a non-constant meromorphic function $f$ and an L-function  $\mathcal{L}$ we have $P(f)=P(\mathcal{L})\implies f=\mathcal{L}$ when $n\geq 2m+4$. Thus $P(z)$ satisfies the condition (i) of {\em\bf Theorem \ref{t1}} of the present paper and  hence $E_f(S, t)=E_{\mathcal{L}}(S,t)\implies f=\mathcal{L}$ for
	\begin{enumerate}
		\item $n\geq\max\{4m+4,2m+7\}\geq 9$ when $t \geq 2$,
		\item  $n\geq\max\{4m+4,2m+8\}\geq 10$ when $t = 1$, and
		\item  $n\geq\max\{4m+4,2m+13\}\geq 15$ when $t = 0$.
	\end{enumerate}
	Again if $f$ is a non-constant meromorphic function with finitely many poles and $\mathcal{L}$ be a non-constant $L$ function, then by {\em\bf Corollary \ref{c1}} we have $E_f(S, t)=E_\mathcal{L}(S,t)\implies f=\mathcal{L}$ for  
	\begin{enumerate}
		\item $n \geq \max\{ 4m+3,2m+5\}=2m+5 \geq 7$ when $t\geq 2 $,
		\item  $n \geq \max\{ 4m+3,2m+6\}=2m+6 \geq 8$ when $t=1 $, and
		\item  $n \geq \max\{ 4m+3,2m+9\}=2m+9 \geq 11$ when $t=0 $.
	\end{enumerate} 
	
\end{exm}

\begin{rem}
	Note that the polynomial \be\label{efr}P(z)=\frac{(n-1)(n-2)}{2}z^n-n(n-2)z^{n-1}+\frac{n(n-1)}{2}z^{n-2}-c,\ee
	where $n\geq6,$ $c\neq0,1$, introduced by Frank-Reinders \cite{frank1998unique} comes as the special case of (\ref{e44.0}) for $m=1$, $a=\frac{-2n}{n-1}$, $b=\frac{n}{n-2}$ and $c\in\mathbb{C}-\{0,\frac{-1}{(n-1)(n-2)}\}$. Hence $E_f(S, t)=E_{\mathcal{L}}(S,t)\implies f=\mathcal{L}$ as $n\geq 9$ when $t \geq 2$, $n\geq 10$ when $t =1$ and $n\geq 15$ when $t = 0$, where $S$ denotes the set of zeros of (\ref{efr}) and $f$, $\mathcal{L}$ are non-constant meromorphic function  and a non-constant L-function respectively.
\end{rem}
\begin{exm}
Let

\begin{equation*}
	P(z) = z^n-4z^{n-1}+4z^{n-2}+c,
\end{equation*}
where $n(\geq 5)$ is odd,
$c\in \mathbb{C}$ such that $P(z)$ does not have any multiple zero. Also we have 
\bea\nonumber P'(z)=z^{n-3}\left(nz^2-4(n-1)z+4(n-2)\right).\eea
Here $P(z)$ is a non-critically injective polynomial \cite[see  (1.1)]{smtbilisi} and we see that $$p=2,\;\; s=3.$$ Suppose $S=\{z:P(z)=0\}$. 
Let $f$ and $\mathcal{L}$ be two non-constant meromorphic  and L-function respectively such that $$P(f) = P(\mathcal{L}).$$ Since $\mathcal{L}$ has at most one pole in $\mathbb{C}$, hence proceeding in the same line of proof as done in {\em Case-3 of Theorem 1.3} of \cite{smtbilisi} or {\em Example 4.4} of \cite{smfil21}  we also get here $f=\mathcal{L}$. 

Therefore $P(z)$ satisfies condition (i) of {\em\bf Theorem \ref{t1}}. Hence we conclude that $E_f(S, t)=E_\mathcal{L}(S,t)\implies f=\mathcal{L}$ when  
\begin{enumerate}
	\item $n\geq \max\{2.2+2,2.3+5\}=11$ for $t\geq 2$, 
	\item $n\geq \max\{2.2+2,2.3+6\}=12$ for $t=1$, and
	\item $n\geq \max\{2.2+2,2.3+11\}=17$ for $t=0$.
\end{enumerate}

Now let  $f$ be a non-constant meromorphic function with finitely many poles and $\mathcal{L}$ be a non-constant $L$ function. Then by {\em\bf Corollary \ref{c1}} we have $E_f(S, t)=E_\mathcal{L}(S,t)\implies f=\mathcal{L}$ for  
\begin{enumerate}
	\item $n \geq \max\{ 2.2+1,2.3+3\}=9$ when $t\geq 2 $,
	\item  $n \geq \max\{ 2.2+1,2.3+4\}=10$ when $t=1 $, and
	\item  $n \geq \max\{ 2.2+1,2.3+7\}=13$ when $t=0 $.
\end{enumerate} 
\end{exm}	 

\begin{exm}
Consider the polynomial 
\bea P(z)=az^n-n(n-1)z^2+2n(n-2)bz -(n-1)(n-2)b^2,\nonumber\eea
where $n(\geq6)$ is an integer and $a$, $b$ are two non-zero complex numbers satisfying
$ab^{n-2}\ne 1, 2$. In \cite[see section 5]{bl}, we can find that this polynomial is still uncertain to be critically injective or non-critically injective. Suppose $S=\{z:P(z)=0\}$.  It is obvious that $n(n-1)z^2-2n(n-2)bz +(n-1)(n-2)b^2=0;$ has two
distinct roots, say $\alpha_1$ and $\alpha_2$. 
Here  \bea\label{e41} R(z)=\frac{az^n}{n(n-1)(z-\alpha_1)(z-\alpha_2)}. \eea
Hence $S=\{z:R(z)-1=0\}$.
From (\ref{e41}) we have
\bea\nonumber R'(z)=\frac{(n-2)az^{n-1}(z-b)^2}{n(n-1)(z-\alpha_1)^2(z-\alpha_2)^2}.\eea
Let $f$  a non-constnat meromorphic function  and $\mathcal{L}$ be a non-constant L-function. Since every L-function is meromorphic in $\mathbb{C}$, so  $R(f)=R(\mathcal{L})\implies f=\mathcal{L}$ for $n\geq6$ directly follows from \cite[see page 67]{aljahari}.
We also find that in this case $$l=2,\;\;k=2. $$ Since $P(z)$ satisfies condition (a) and (i) of  {\em\bf Theorem \ref{t2}}. Hence we obtain $E_f(S, t)=E_\mathcal{L}(S,t)\implies f=\mathcal{L}$ for \begin{enumerate}
	\item $n\geq \max\{2k+4,2l+5\}=9$ when $t\geq2$,
	\item $n\geq \max\{2k+4,2l+6\}=10$ when $t=1$, and 
	\item for $n\geq \max\{2k+4,2l+11\}=15$ when $t=0$.
\end{enumerate}

Again if $f$ is a non-constant meromorphic function with finitely many poles and $\mathcal{L}$ is a non-constant L-function, then by {\em\bf Corollary \ref{c2}} we have $E_f(S, t)=E_\mathcal{L}(S,t)\implies f=\mathcal{L}$ for  
\begin{enumerate}
	\item $n \geq \max\{ 2k+3,2l+3\}=7$ when $t\geq 2 $,
	\item  $n \geq \max\{ 2k+3,2l+4\}=8$ when $t=1 $, and
	\item  $n \geq \max\{ 2k+3,2l+7\}=11$ when $t=0 $.
\end{enumerate} 
\end{exm}
\section{Lemmas}
In this section we prove some lemmas which will be required for the proofs of our main results. Before stating these lemmas we shortly recall the following definitions for the convenience of our readers. 
\begin{defi} \cite{15} Let $f$ and $g$ be two non-constant meromorphic functions such that $f$ and $g$ share $(1,0)$. Let $z_{0}$ be a $1$-point of $f$ with multiplicity $p$, a $1$-point of $g$ with multiplicity $q$. We denote by $\ol N_{L}(r,1;f)$ the reduced counting function of those $1$-points of $f$ and $g$ where $p>q$, by $N^{1)}_{E}(r,1;f)$ the counting function of those $1$-points of $f$ and $g$ where $p=q=1$. In the same way we can define $\ol N_{L}(r,1;g)$, $N^{1)}_{E}(r,1;g)$. In a similar manner we can define $\ol N_{L}(r,a;f)$ and $\ol N_{L}(r,a;g)$ for $a\in\mathbb{\ol C}$. \end{defi}
\begin{defi}\cite{lahiriNagoya, Lahiri CVEE} Let $f$, $g$ share $(a,0)$. We denote by $\ol N_{*}(r,a;f,g)$ the reduced counting function of those $a$-points of $f$ whose multiplicities differ from the multiplicities of the corresponding $a$-points of $g$. 
\end{defi}
Clearly $$\ol N_{*}(r,a;f,g)\equiv\ol N_{*}(r,a;g,f)\;\;\mbox{and}\;\;\ol N_{*}(r,a;f,g)=\ol N_{L}(r,a;f)+\ol N_{L}(r,a;g).$$ 

\begin{defi} \cite{Lahiri CVEE} For $a\in\mathbb{C}\cup\{\infty\}$ we denote by $N(r,a;f\vline=1)$ the counting function of simple $a$ points of $f$. For a positive integer $m$ we denote by $N(r,a;f\vline\leq m) (N(r,a;f\vline\geq m))$ the counting function of those $a$ points of $f$ whose multiplicities are not greater(less) than $m$ where each $a$ point is counted according to its multiplicity.

$\ol N(r,a;f\vline\leq m)\; (\ol N(r,a;f\vline\geq m))$ are defined similarly, where in counting the $a$-points of $f$ we ignore the multiplicities.

Also $N(r,a;f\vline <m),\; N(r,a;f\vline >m),\; \ol N(r,a;f\vline <m)\; and\; \ol N(r,a;f\vline >m)$ are defined analogously.  \end{defi}
\begin{defi}\cite{ab1} Let $a,b_{1},b_{2},\ldots,b_{q} \in\mathbb{C}\;\cup\{\infty\}$. We denote by $N(r,a;f\vline\; g\neq b_{1},b_{2},\ldots,b_{q})$ the counting function of those $a$-points of $f$, counted according to multiplicity, which are not the $b_{i}$-points of $g$ for $i=1,2,\ldots,q$. \end{defi}
For  two non-constant meromorphic functions $F$ and $G$, 
set\be{\label{e2.2}}H=\left(\frac{F^{''}}{F^{'}}-\frac{2F^{'}}{F-1}\right)-\left(\frac{G^{''}}{G^{'}}-\frac{2G^{'}}{G-1}\right).\ee

\begin{lem}\label{2.1}\cite{15}
Let $F$, $G$  share $(1,0)$ and $H\not\equiv 0$. Then $$N_{E}^{1)}(r,1;F)=N_{E}^{1)}(r,1;G)\leq N(r,H)+S(r,F)+S(r,G).$$
\end{lem}

\begin{lem}\label{102}
Let $f$ be a non-constant meromorphic function and $\mathcal{L}$ be an non-constant L-function sharing a set $S$ IM, where $|S|\geq3$. Then $\rho(f)=\rho(\mathcal{L})=1$. Furthermore, $ S(r,f)=O(\log r)=S(r,\mathcal{L}).$ 
\end{lem}
\begin{proof}
Proceeding in a similar method as done in the proof of Theorem 5, \cite[see p. 6]{yuan2018}, we would obtain $\rho(f)=\rho(\mathcal{L})=1$. So we omit it.\\
Since $\rho(f)=\rho(\mathcal{L})=1$, so from the definition of $S(r,f)$ we get $S(r,f)=O(\log r)=S(r,\mathcal{L})$.
\end{proof}
\begin{lem}\label{new}
Let $f$, $g$ be two non-constant meromorphic functions sharing (1,t), where $t\in\mathbb{N}\cup\{0\}$. Then
$$\ol N(r,1;f)+\ol N(r,1;g)\leq N_{E}^{1)}(r,1;f)+\left(1-t\right) \ol N_{*}(r,1;f,g)+N(r,1;f).$$

\begin{proof}
Since $f$ and $g$ share $(1,t)$, we observe that $$\ol N(r,1;f)+\ol N(r,1;g)=2\ol N(r,1;f).$$
{\bf Case-I :} Suppose $t\geq 2$.\\
Let $z_0$ be a $1$ point of $f$ with multiplicity $p$ and a $1$ point of $g$ of multiplicity $q$. Since $f$ and $g$ share $(1,t)$, therefore $p\leq t$ implies $p=q$.
\par{\bf Subcase - I :} When $p\leq t$.
If $p=1$, then $z_0$ is counted once in both $N_{E}^{1)}(r,1;f)$ and $N(r,1;f)$. On the other hand $z_0$ is not counted in $\ol N_{*}(r,1;f,g)$. Again if $p\ne 1$, then $z_0$ is counted $p$ times (i.e., at least 2 times) in $N(r,1;f)$ and in this case $z_0$ is not counted in $N_{E}^{1)}(r,1;f) \;\text{and}\; \ol N_{*}(r,1;f,g)$. Therefore $z_0$ is counted at least $2$ times in $N_{E}^{1)}(r,1;f)+\left(1-t\right) \ol N_{*}(r,1;f,g)+N(r,1;f).$
\par{\bf Subcase - II :} When $p\geq (t+1)$. If $p=q$, then $z_0$ is counted $p$ time (i.e., at least 3 times) in  $N(r,1;f)$ and $z_0$ is not counted in $N_{E}^{1)}(r,1;f) \;\text{and}\; \ol N_{*}(r,1;f,g)$. When $p\ne q$, then $z_0$ is counted $(1-t)$ times in $(1-t)\ol N_{*}(r,1;f,g)$ and counted $p$ times (i.e., at least $t+1$ times) in  $N(r,1;f)$ and $z_0$ is not counted in $ N_{E}^{1)}(r,1;f)$; i.e., $z_0$ is counted at least $(1-t)+(t+1)=2$ times in $N_{E}^{1)}(r,1;f)+\left(1-t\right) \ol N_{*}(r,1;f,g)+N(r,1;f).$\\
Now since $z_0$ is counted two times in $\ol N(r,1;f)+\ol N(r,1;g).$ Therefore in any sub case we have 
$$\ol N(r,1;f)+\ol N(r,1;g)\leq N_{E}^{1)}(r,1;f)+\left(1-t\right) \ol N_{*}(r,1;f,g)+N(r,1;f) .$$
{\bf Case-II :} Suppose $t=1$.\\
Then clearly \bea 2\ol N(r,1;f)\nonumber &\leq& N(r,1;f|=1) +N(r,1
;f)\\\nonumber&=&  N_{E}^{1)}(r,1;f)+N(r,1;f).\nonumber \eea
Therefore, for $t=1$
$$\ol N(r,1;f)+\ol N(r,1;g)\leq N_{E}^{1)}(r,1;f)+\left(1-t\right) \ol N_{*}(r,1;f,g)+N(r,1;f).$$

{\bf Case-III :} Suppose that $t=0$.\\
Let $z_0$ be a $1$ point of $f$ with multiplicity $p$ and a $1$ point of $g$ of multiplicity $q$. If $p=q=1$, then $z_0$ is counted $2$ times in both $2\ol N(r,1;f)$ and $N_{E}^{1)}(r,1;f)+\left(1-t\right) \ol N_{*}(r,1;f,g)+N(r,1;f)$, as $\ol N_{*}(r,1;f,g)$ does not count $z_0$. If $p=1,\;q\neq 1$, then also $z_0$ is counted $2$ times in both $2\ol N(r,1;f)$ and\\ $N_{E}^{1)}(r,1;f)+\left(1-t\right)\ol N_{*}(r,1;f,g)+N(r,1;f)$, as in this case $N_{E}^{1)}(r,1;f)$ does not count $z_0$. Finally if $p\neq 1$, then $z_0$ is counted at least 2 times in $\left(1-t\right) \ol N_{*}(r,1;f,g)+N(r,1;f)$ and $z_0$ is not counted in $N_{E}^{1)}(r,1;f)$.\\
Therefore, for $t=0$, $$\ol N(r,1;f)+\ol N(r,1;g)\leq N_{E}^{1)}(r,1;f)+\left(1-t\right) \ol N_{*}(r,1;f,g)+N(r,1;f).$$
And hence in any case,	$$\ol N(r,1;f)+\ol N(r,1;g)\leq N_{E}^{1)}(r,1;f)+\left(1-t\right) \ol N_{*}(r,1;f,g)+N(r,1;f).$$

\end{proof}
\end{lem}

\begin{lem}
Let $F^{*}-1=\d\frac{a_{n}\prod\limits_{i=1}^{n}(f-w_{i})}{\psi(f)}$ and $G^{*}-1=\d\frac{a_{n}\prod\limits_{i=1}^{n}(\mathcal{L}-w_{i})}{\psi(\mathcal
{L})}$, where $f$  be a non-constant meromorphic function, $\mathcal{L}$ be an non-constant L-function, $a_{n},w_{i}\in\mathbb{C}-\{0\}$; $\forall i\in\{1,2,\ldots,n\}$ and  $\psi(z)$ be a polynomial of degree less than $n$ with $\psi(w_{i})\neq0$; $\forall i\in\{1,2,\ldots,n\}$. Further suppose that $F^{*}$ and $G^{*}$  share $(1,t)$, where $t\in\mathbb{N}\cup\{0\}$.  Then 
\be\nonumber \ol{N}_L(r,1;F^{*})\leq \frac{1}{t+1}\left[\ol{ N}(r,\infty;f)+\ol{ N}(r,0;f)-N_1(r,0,f')\right]+O(\log r),\ee
where $N_1(r,0,f')=N(r,0;f'|f\ne 0,w_1,w_2,...,w_n).$
Similar expression also holds for $\ol  N_L(r,1;G^{*})$.
\end{lem}
\begin{proof}
Since $F^{*}$ and $G^{*}$  share $(1,t)$, so using the first fundamental theorem we find that
\bea  \ol{ N}_L(r,1;F^{*}) &\leq& \ol{ N}(r,1;F^{*}|\geq t+2)\nonumber\\&&
\leq \frac{1}{t+1}\left[ N(r,1;F^{*})-\ol{ N}(r,1;F^{*})\right]
\nonumber\\&&
\leq \frac{1}{t+1}\left[\sum_{i=1}^{n}\left(N(r,w_i;f)-\ol{ N}(r,w_i;f)\right)\right]	\nonumber\\&&
\leq \frac{1}{t+1}\left[N(r,0;f'|f\ne 0)-N_1(r,0,f')\right]\nonumber\\&&
\leq \frac{1}{t+1}\left[N(r,0;\frac{f'}{f})-N_1(r,0,f')\right]\nonumber\\&&
\leq \frac{1}{t+1}\left[N(r,\infty;\frac{f}{f^{'}})-N_1(r,0,f')\right]+O(\log r)\nonumber\\&&
\leq \frac{1}{t+1}\left[N(r,\infty;\frac{f^{'}}{f})-N_1(r,0,f')\right]+O(\log r)\nonumber\\&&
\leq \frac{1}{t+1}\left[\ol{ N}(r,\infty;f)+\ol{ N}(r,0;f)-N_1(r,0,f')\right]+O(\log r)\nonumber\
\eea 
This proves the lemma.
\end{proof}

\begin{lem}\label{l11}

Let $P(z)$, $S$ and $s$ as defined by (\ref{el000}), (\ref{el00}) and (\ref{el0000}) respectively. Suppose that $f$, $\mathcal{L}$ share $(S,t)$, where  $t\in \mathbb{N}\cup\{0\}$ and $f$, $\mathcal{L}$ be a non-constant meromorphic function and an $L$-function respectively. Further suppose that  \be\label{el04}\mathcal{F}=\d\frac{P(f)-a_{0}}{-a_{0}}\;\; and\;\; \mathcal{G}=\d\frac{P(\mathcal{L})-a_{0}}{-a_{0}}.\ee Then for $n\geq 2s+5$, when $t\geq 2$ ,$n\geq 2s+6$, when $t=1$ and for $n\geq2s+11$, when $t=0$ we get the following.
\beas\frac{1}{\mathcal{F}-1}=\frac{A}{\mathcal{G}-1}+B,\eeas where $A(\neq0), B\in\mathbb{C}$.
\end{lem}
\begin{proof}
According to the assumptions of the lemma we clearly have  $\mathcal{F}$, $\mathcal{G}$ share $(1,t)$ and $$\mathcal{F}^{'}=-\frac{na_{n}}{a_{0}}\prod\limits_{i=1}^{s}(f-\eta_{i})^{m_{i}}f^{'};\;\;\; \mathcal{G}^{'}=-\frac{na_{n}}{a_{0}}\prod\limits_{i=1}^{s}(\mathcal{L}-\eta_{i})^{m_{i}}\mathcal{L}^{'},$$ where $\sum\limits_{i=1}^{s}m_{i}=n-1$. 
Now consider $H$ as given by (\ref{e2.2}) for $\mathcal{F}$ and $\mathcal{G}$.
\vspace{0.1in}\par{\bf Case-I:} Suppose $H\not\equiv 0$. Then, it can be easily verified that $H$ has only simple poles and these poles come  from  the following points.
\begin{enumerate}
\item[(i)] $\eta_{i}$ points of $f$ and $\mathcal{L}$.
\item[(ii)] Poles of $f$ and $\mathcal{L}$.
\item[(iii)] $1$ points of $\mathcal{F}$ and $\mathcal{G}$ having different multiplicities.
\item[(iv)] Those zeros of $f^{'}$ and $\mathcal{L}^{'}$ which are not zeros of   $\prod\limits_{i=1}^{s}(f-\eta_{i})(\mathcal{F}-1)$ and $\prod\limits_{i=1}^{s}(\mathcal{L}-\eta_{i})(\mathcal{G}-1)$ respectively.
\end{enumerate}
Therefore we obtain 
\bea\label{e2.1} N(r,H)&\leq& \sum\limits_{i=1}^{s}\left[\ol N(r,\eta_{i};f)+\ol N(r,\eta_{i};\mathcal{L})\right]+\ol N(r,\infty;f)+\ol N(r,\infty;\mathcal{L})\\\nonumber&&+\ol N_{*}(r,1;\mathcal{F},\mathcal{G})+\ol N_{0}(r,0,f^{'})+\ol N_{0}(r,0,\mathcal{L}^{'}),\eea where $\ol N_{0}(r,0,f^{'})$ and $\ol N_{0}(r,0,\mathcal{L}^{'})$ denotes the reduced counting functions of those zeros of $f^{'}$ and $\mathcal{L^{'}}$ which are not zeros of   $\prod\limits_{i=1}^{s}(f-\eta_{i})(\mathcal{F}-1)$ and $\prod\limits_{i=1}^{s}(\mathcal{L}-\eta_{i})(\mathcal{G}-1)$ respectively.
Using the second fundamental theorem we get 
\bea\label{e2.3}&&(n+s-1)T(r,f)\\\nonumber&\leq&\ol N(r,1;\mathcal{F})+\sum\limits_{i=1}^{s}\ol N(r,\eta_{i};f)+\ol N(r,\infty;f)-N_{0}(r,0,f^{'})+S(r,f),\eea
\bea\label{e2.4}&&(n+s-1)T(r,\mathcal{L})\\\nonumber&\leq&\ol N(r,1;\mathcal{G})+\sum\limits_{i=1}^{s}\ol N(r,\eta_{i};\mathcal{L})+\ol N(r,\infty;\mathcal{L})-N_{0}(r,0,\mathcal{L}^{'})+S(r,\mathcal{L}).\eea
Let us denote by $T(r)=\text{max}\{T(r,f),T(r,\mathcal{L})\}$. Now combining (\ref{e2.3}) and (\ref{e2.4}) and using {\em Lemma} {\bf\ref{102}} and the fact that $\ol{N}(r,\infty;\mathcal{L})=O(\log r)$, we get
\bea
\label{e2.5}&&(n+s-1)\{T(r,f)+T(r,\mathcal{L})\}\\\nonumber&\leq&\ol N(r,1;\mathcal{F})+\ol N(r,1;\mathcal{G})+\sum\limits_{i=1}^{s}\left[\ol N(r,\eta_{i};f)+\ol N(r,\eta_{i};\mathcal{L})\right]\\\nonumber&&+\ol N(r,\infty;f)-N_{0}(r,0,f^{'})-N_{0}(r,0,\mathcal{L}^{'})+O(\log r).\nonumber
\eea


Now it follows from {\em Lemma} {\bf\ref{new}}, {\em Lemma} {\bf\ref{102}}, {\em Lemma} {\bf\ref{2.1}}, (\ref{e2.1}) and in view of $\ol{N}(r,\infty;\mathcal{L})=O(\log r)$ that 
\bea\label{eq11.4}&& \ol N(r,1;\mathcal{F})+\ol N(r,1;\mathcal{G})\\\nonumber&\leq&  (s+1)T(r,f)+sT(r,\mathcal{L})+(2-t)\ol N_{*}(r,1;\mathcal{F},\mathcal{G})+N(r,1;\mathcal{F})\\\nonumber&&+\ol N_{0}(r,0,f^{'})+\ol N_{0}(r,0,\mathcal{L}^{'})+O(\log r).\eea
and \bea\label{eq11.5}&&\ol N(r,1;\mathcal{F})+\ol N(r,1;\mathcal{G})\\\nonumber&\leq&  (s+1)T(r,f)+sT(r,\mathcal{L})+(2-t)\ol N_{*}(r,1;\mathcal{F},\mathcal{G})+N(r,1;\mathcal{G})\\\nonumber&&+\ol N_{0}(r,0,f^{'})+\ol N_{0}(r,0,\mathcal{L}^{'})+O(\log r).\eea\\

Now from (\ref{e2.5}) , (\ref{eq11.4}) and using $T(r,\mathcal{F})=nT(r,f)$ we get
\bea(n+s-1)\{T(r,f)+T(r,\mathcal{L})\}\nonumber&\leq&(2s+2)T(r,f)+2sT(r,\mathcal{L})\\\nonumber&&+(2-t)\ol N_{*}(r,1;\mathcal{F},\mathcal{G})+nT(r,f)+O(\log r);\eea
i.e.,
\bea\label{eq11.6}&& nT(r,\mathcal{L})\\\nonumber
&\leq& (s+3)T(r,f)+(s+1)T(r,\mathcal{L})+(2-t)\ol N_{*}(r,1;\mathcal{F},\mathcal{G})+O(\log r).\\\nonumber&
.\nonumber\eea

Also from (\ref{e2.5}) , (\ref{eq11.5}) and using $T(r,\mathcal{G})=nT(r,\mathcal{L})$ we get

\bea(n+s-1)\{T(r,f)+T(r,\mathcal{L})\}\nonumber&\leq&(2s+2)T(r,f)+2sT(r,\mathcal{L})+\\\nonumber&&(2-t)\ol N_{*}(r,1;\mathcal{F},\mathcal{G})+nT(r,\mathcal{L})+O(\log r);\eea
i.e.,
\bea\label{eq11.7}&& nT(r,f)\\\nonumber
&\leq& (s+3)T(r,f)+(s+1)T(r,\mathcal{L})+(2-t)\ol N_{*}(r,1;\mathcal{F},\mathcal{G})+O(\log r).\\\nonumber&
.\nonumber\eea
Therefore, from (\ref{eq11.6}) and (\ref{eq11.7}) we get,

\be\label{11.8}nT(r)\leq (2s+4)T(r)+(2-t)\ol N_{*}(r,1;\mathcal{F},\mathcal{G})+O(\log r).\ee
Now for $t\geq2$; from (\ref{11.8}) we get$$nT(r)\leq (2s+4)T(r)+O(\log r),$$
which is a contradiction for $n\geq 2s+5$.\\          
When $t <2$, from (\ref{11.8}) we get
\bea\label{eq11.9} nT(r)
&\leq& (2s+4)T(r)+(2-t)\ol N_{*}(r,1;\mathcal{F},\mathcal{G})+O(\log r)\\\nonumber&
\leq& (2s+4)T(r)+\frac{(2-t)}{t+1}[\ol{ N}(r,\infty;f)+\ol{ N}(r,0;f)\\\nonumber&&+\ol{ N}(r,\infty;\mathcal{L})+\ol{ N}(r,0;\mathcal{L})]+O(\log r) \nonumber\\&\leq& \left(2s+4+\frac{3(2-t)}{t+1}\right)T(r)+O(\log r)
.\nonumber\eea

Now for $t=1$; from (\ref{eq11.9}) we get$$nT(r)\leq (2s+4+\frac{3}{2})T(r)+O(\log r),$$
which is a contradiction for $n\geq 2s+6$.\\  
For $t=0$; from (\ref{eq11.9}) we get $$nT(r)\leq (2s+4+6)T(r)+O(\log r),$$
which is a contradiction for $n\geq 2s+11$.\\

\vspace{0.1in}\par{\bf Case-II:} Suppose	$H\equiv0$. Now integrating (\ref{e2.2}), we find that \bea\nonumber \frac{1}{\mathcal{F}-1}=\frac{A}{\mathcal{G}-1}+B,\;\;\; \mbox{where}\;\; A(\neq0), B\in\mathbb{C}.\eea 	
\end{proof}

\begin{lem}\label{l1.1}
Let $R(z)$, $S$ and $l$ as defined by (\ref{el2}), (\ref{el02}) and (\ref{el3}) respectively. Suppose that $f$, $\mathcal{L}$ share $(S,t)$, where  $t\in \mathbb{N}\cup\{0\}$ and $f$, $\mathcal{L}$ be a non-constant meromorphic function and an $L$-function respectively. Further suppose that \be\label{el05}\mathbb{F}= R(f)\;\; and\;\; \mathbb{G}=R(\mathcal{L}).\ee 
Then for $n\geq 2l+5$, when $t\geq2$, for $n\geq 2l+6$, when $t=1$ and for $n\geq 2l+11$, when $t=0$ we get the following.
\beas\frac{1}{\mathbb{F}-1}=\frac{A}{\mathbb{G}-1}+B,\eeas where $A(\neq0), B\in\mathbb{C}$.
\end{lem}
\begin{proof} 	
Clearly $\mathbb{F}, \mathbb{G}$ share $(1,t)$ and in view of (\ref{el3}) we have
\be \nonumber \mathbb{F}^{'}=\frac{\gamma\prod\limits_{j=1}^{l}(f-\delta_j)^{q_{j}}}{\prod\limits_{j=1}^{k}(f-\beta_j)^{p_{j}}}f',\;\;\;\;\;\;\mathbb{G}^{'}=\frac{\gamma\prod\limits_{j=1}^{l}(\mathcal{L}-\delta_j)^{q_{j}}}{\prod\limits_{j=1}^{k}(\mathcal{L}-\beta_j)^{p_{j}}}\mathcal{L}'.\ee 
Now consider $H$ as given by (\ref{e2.2}) for $\mathbb{F}$ and $\mathbb{G}$.
\vspace{0.1in}\par{\bf Case-I:} Suppose $H\not\equiv 0$. Since $H$ has only simple poles and in this case these poles come  from  the following points.
\begin{enumerate}
\item[(i)] $\delta_{j}$ -points of $f$ and $\mathcal{L}$.
\item[(ii)] Poles of $f$ and $\mathcal{L}$.
\item[(iii)] $1$ points of $\mathbb{F}$ and $\mathbb{G}$ having different multiplicities.
\item[(iv)] Those zeros of $f^{'}$ and $\mathcal{L^{'}}$ which are not zeros of   $\prod\limits_{j=1}^{l}(f-\delta_{j})(\mathbb{F}-1)$ and $\prod\limits_{j=1}^{l}(\mathcal{L}-\delta_{j})(\mathbb{G}-1)$ respectively.
\end{enumerate}
Therefore we obtain 
\bea\label{ep2} N(r,H) &\leq& \ol{N}(r,\infty;f)+\sum_{j=1}^{l}\ol{N}(r,\delta_j;f)+\ol N_0(r,0;f') +\ol{N}(r,\infty;\mathcal{L})\\\nonumber&&+\sum_{j=1}^{l}\ol{N}(r,\delta_j;\mathcal{L})+\ol N_0(r,0;\mathcal{L}')+\ol N_*(r,1;\mathbb{F},\mathbb{G})\\\nonumber&&+S(r,f)+S(r,\mathcal{L}),\eea
where we write $\ol N_0(r,0;f')$ for the reduced counting function of the zeros of $f'$ that are not zeros
of $(\mathbb{F}-1)\prod\limits_{j=1}^{l}(f-\delta_j)^{q_{j}}$ and $\ol N_0(r,0;\mathcal{L}')$ is similarly defined. Now using the second fundamental theorem we get
\bea\label{e3.3}&&(n+l-1)T(r,f)\\\nonumber&\leq&\ol N(r,1;\mathbb{F})+\sum\limits_{i=1}^{l}\ol N(r,\delta_{i};f)+\ol N(r,\infty;f)-N_{0}(r,0,f^{'})+S(r,f),\eea
\bea\label{e3.4}&&(n+l-1)T(r,\mathcal{L})\\\nonumber&\leq&\ol N(r,1;\mathbb{G})+\sum\limits_{i=1}^{l}\ol N(r,\delta_{i};\mathcal{L})+\ol N(r,\infty;\mathcal{L})-N_{0}(r,0,\mathcal{L}^{'})+S(r,\mathcal{L}).\eea
Let us denote by $T(r)=\text{max}\{T(r,f),T(r,\mathcal{L})\}$. 	 Now combining (\ref{e3.3}), (\ref{e3.4}) and using {\em Lemma} {\bf\ref{102}} and the fact that $\ol{N}(r,\infty;\mathcal{L})=O(\log r)$, we get
\bea\label{e3.5}&&(n+l-1)\{T(r,f)+T(r,\mathcal{L})\}\\\nonumber&\leq&\ol N(r,1;\mathbb{F})+\ol N(r,1;\mathbb{G})+\sum\limits_{i=1}^{l}\left[\ol N(r,\delta_{i};f)+\ol N(r,\delta_{i};\mathcal{L})\right]\\\nonumber&&+\ol N(r,\infty;f)-N_{0}(r,0,f^{'})-N_{0}(r,0,\mathcal{L}^{'})+O(\log r).\eea 	  		
Now it follows from {\em Lemma} {\bf\ref{new}}, {\em Lemma} {\bf\ref{102}}, {\em Lemma} {\bf\ref{2.1}} and (\ref{ep2}) and in view of $\ol{N}(r,\infty;\mathcal{L})=O(\log r)$ that 
\bea \label{eq12.4} &&\ol N(r,1;\mathbb{F})+\ol N(r,1;\mathbb{G})\\\nonumber&\leq&  (l+1)T(r,f)+lT(r,\mathcal{L})+(2-t)\ol N_{*}(r,1;\mathbb{F},\mathbb{G})+N(r,1;\mathbb{F})\\\nonumber&&+\ol N_{0}(r,0,f^{'})+\ol N_{0}(r,0,\mathcal{L}^{'})+O(\log r).\eea
and \bea\label{eq12.5}&& \ol N(r,1;\mathbb{F})+\ol N(r,1;\mathbb{G})\\\nonumber&\leq& (l+1)T(r,f)+lT(r,\mathcal{L})+(2-t)\ol N_{*}(r,1;\mathbb{F},\mathbb{G})\\\nonumber&&+N(r,1;\mathbb{G})+\ol N_{0}(r,0,f^{'})+\ol N_{0}(r,0,\mathcal{L}^{'})+O(\log r).\eea 	  	
Now from (\ref{e3.5}), (\ref{eq12.4}) and using $T(r,\mathbb{F})=nT(r,f)$ we get
\bea(n+l-1)\{T(r,f)+T(r,\mathcal{L})\}\nonumber&\leq&(2l+2)T(r,f)+2lT(r,\mathcal{L})+(2-t)\ol N_{*}(r,1;\mathbb{F},\mathbb{G})\\\nonumber&&+nT(r,f)+O(\log r);\eea
i.e.,
\bea\label{eq12.6} &&nT(r,\mathcal{L})\\\nonumber
&\leq& (l+3)T(r,f)+(l+1)T(r,\mathcal{L})+(2-t)\ol N_{*}(r,1;\mathbb{F},\mathbb{G})+O(\log r).\\\nonumber&
\nonumber\eea 	   
Also from (\ref{e3.5}), (\ref{eq12.5}) and using $T(r,\mathbb{G})=nT(r,\mathcal{L})$ we get 	   
\bea(n+l-1)\{T(r,f)+T(r,\mathcal{L})\}\nonumber&\leq&(2l+2)T(r,f)+2lT(r,\mathcal{L})+(2-t)\ol N_{*}(r,1;\mathbb{F},\mathbb{G})\\\nonumber&&+nT(r,\mathcal{L})+O(\log r);\eea
i.e.,
\bea\label{eq12.7}&& nT(r,f)\\\nonumber
&\leq& (l+3)T(r,f)+(l+1)T(r,\mathcal{L})+(2-t)\ol N_{*}(r,1;\mathbb{F},\mathbb{G})+O(\log r).\\\nonumber&
\nonumber\eea
Therefore, from (\ref{eq12.6}) and (\ref{eq12.7}) we get,

\be\label{12.8}nT(r)\leq (2l+4)T(r)+(2-t)\ol N_{*}(r,1;\mathbb{F},\mathbb{G})+O(\log r).\ee
Now for $t\geq2$; from (\ref{12.8}) we get$$nT(r)\leq (2l+4)T(r)+O(\log r),$$
which is a contradiction for $n\geq 2l+5$.\\          
When $t <2$, from (\ref{12.8}) we get
\bea\label{eq12.9} &&nT(r)\\\nonumber
&\leq& (2l+4)T(r)+(2-t)\ol N_{*}(r,1;\mathbb{F},\mathbb{G})+O(\log r)\\\nonumber&
\leq& (2l+4)T(r)+ \frac{(2-t)}{t+1}\ol{ N}(r,\infty;f)+\ol{ N}(r,0;f)+\ol{ N}(r,\infty;\mathcal{L})\\\nonumber&&+\ol{ N}(r,0;\mathcal{L})+O(\log r) \nonumber\\&\leq& \left(2l+4+\frac{3(2-t)}{t+1}\right)T(r)+O(\log r)
.\nonumber\eea

Now for $t=1$; from (\ref{eq12.9}) we get$$nT(r)\leq (2l+4+\frac{3}{2})T(r)+O(\log r),$$
which is a contradiction for $n\geq 2l+6$.\\  
For $t=0$; from (\ref{eq12.9}) we get $$nT(r)\leq (2l+4+6)T(r)+O(\log r),$$
which is a contradiction for $n\geq 2l+11$.\\


\vspace{0.1in}\par{\bf Case-II:} Suppose	$H\equiv0$. Now integrating (\ref{e2.2}), we find that \bea\nonumber \frac{1}{\mathbb{F}-1}=\frac{A}{\mathbb{G}-1}+B,\;\;\; \mbox{where}\;\; A(\neq0), B\in\mathbb{C}.\eea 
\end{proof}

\begin{lem}\label{l2}\cite{hxy0}
Let $F$ and $G$ be two non-constant meromorphic functions such that\\ \beas\displaystyle\frac{1}{F-1}=\frac{A}{G-1}+B,\eeas where $A(\ne0),B\in\mathbb{C}$. If
\begin{equation*}
\ol{N}(r,0; F)+	\ol{N}(r,\infty; F)+	\ol{N}(r,0; G)+	\ol{N}(r,\infty; G) < T(r),
\end{equation*} 
where $T(r)=\max\{T(r,F),T(r,G)\}$. Then either $FG=1$ or $F=G.$
\end{lem}
\begin{lem}\label{14}
Let $\mathcal{F}$, $\mathcal{G}$ be defined by (\ref{el04}). If $p \geq 2$
Then $\mathcal{F}\mathcal{G} \not= a$, where $a$ is non-zero complex constant.  
\end{lem}
\begin{proof}
On the contrary, suppose that $\mathcal{F}\mathcal{G} = a$. Then
\be\label{e01}\prod_{i=1}^{p}(f-\alpha_i)^{m_i}\prod_{i=1}^{p}(\mathcal{L}-\alpha_i)^{m_i}=a\left(\frac{a_{0}}{a_{n}}\right)^{2}=a_{1}(say).\ee

It is clear from (\ref{e01} ) that each $\alpha_i$-point of $f$ is a pole of $\mathcal{L}$ and vice-versa.\\
{\bf\underline{Case-i:}} Let $p\geq 4$. Since  an L- function has at most one pole, then in view of (\ref{e01}) we can say that $f$ has at least three $\alpha_i$-points which are picard exeptional values. That is, the meromorphic function $f$ omits at least 3 values, so  $f$ must be constant. This contradicts our assumption. \\
{\bf\underline{Case-ii:}} Let $p = 3$. Again like the arguments made above we can say that $f$ omits two values say $\alpha_{1}, \alpha_{2}$. Hence $\alpha_{3}$ points of $f$ are the poles of $\mathcal{L}$. Since $z=1$ is the only pole of $\mathcal{L}$, so let  $z=1$ be the $\alpha_3$-point of $f$ of multiplicity $r$ and the pole of $\mathcal{L}$ of multiplicity $u$. Then $m_3r = nu$. Since $u\geq 1$, so $r \geq \frac{n}{m_{3}}$; i.e., $\frac{1}{r}\leq \frac{m_{3}}{n}$.
Hence using the second fundamental theorem, we obtain 
\begin{flalign*}
T(r, f) & \leq \sum_{i=1}^{3}\ol{N}(r, \alpha_i; f)+S(r,f)\\
& \leq \frac{m_3}{n}{N}(r,\alpha_3; f)+O(\log r),
\end{flalign*}
which is  a  contradiction as $m_3 < n$. \\
{\bf\underline{Case-iii}:} Let $p=2$.\\ Since  an L- function has at most one pole, then in view of (\ref{e01}) we can say that $f$ omits at least one $\alpha_i$-points, say $\alpha_1$.


Let $z_0$ be a $\alpha_2$ point of $\mathcal{L}$ of multiplicity $r$, then it will be a pole of $f$ of multiplicity $\nu$, such that $rm_2=\nu n$. Since $\nu \geq 1$, so $r \geq \frac{n}{m_2}$ ; i.e., $\frac{1}{r} \leq\frac{m_2}{n}.$ Now using the second fundamental theorem  in view of $\ol{N}(r,\infty;\mathcal{L})=O(\log r)$ we get
\beas T(r,\mathcal{L}) &\leq& \ol N(r,\alpha_{1};\mathcal{L})+\ol N(r,\alpha_{2};\mathcal{L})+\ol  N(r,\infty;\mathcal{L})+S(r,\mathcal{L})\\&\leq& \frac{m_2}{n} T(r,\mathcal{L})+O(\log r),\eeas
which is a contradiction as $n>m_2$. 	
\end{proof}

\begin{lem}\label{l04}
Let $\mathbb{F}, \mathbb{G}$ as defined by (\ref{el05}). If
\begin{enumerate}
\item[(a)] $k\geq2$, or
\item[(b)] $k=1$ and $n\geq 3b_1+1$, where $b_1=gcd(m_1,n)$,
\end{enumerate} 
Then $\mathbb{F}\mathbb{G} \not= a$, where $a$ is non-zero complex constant.\end{lem}
\begin{proof}
On the contrary suppose that $\mathbb{F} \mathbb{G}\equiv a$. Then 
\bea\label{000} \d\frac{f^n}{\prod\limits_{j=1}^{k}(f-\beta_j)^{m_{j}}}.\d\frac{\mathcal{L}^n}{\prod\limits_{j=1}^{k}(\mathcal{L}-\beta_j)^{m_{j}}}\equiv a\left(\frac{a_{i}}{a_{n}}\right)^{2}=a^{'}(say),\eea and hence   \bea\label{0000} T(r,f)= T(r,\mathcal{L})+O(1).\eea 		
It is clear from (\ref{000} ) that  $\beta_j$ point of $f$ is a zero  of $\mathcal{L}$ and vice-versa. 

{\bf(a)} Let $k\geq 2$.
If $z_0$ be a zero of $\mathcal{L}-\beta_j$ with
multiplicity $p$, then $z_0$ is a zero of $f$ with multiplicity $q$ such that $m_jp=nq$; i.e., $p\d\geq \frac{n}{m_j}$. Therefore  $\ol{N}(r,\beta_j;\mathcal{L})\leq \d\frac{m_j}{n}N(r,\beta_j;\mathcal{L})$. 

So, using the second fundamental theorem and $\ol{N}(r,\mathcal{L})=O(\log r)$, we get \bea\nonumber (k-1)T(r,\mathcal{L})&\leq& \sum_{j=1}^{k}\ol{N}(r,\beta_j;\mathcal{L})+\ol{N}(r,\infty;\mathcal{L})+S(r,\mathcal{L}) \nonumber\\ & 
\leq& \sum_{j=1}^{k}\frac{m_j}{n}T(r,\mathcal{L})+O(\log r)\nonumber\\&
\leq& (1-\frac{1}{n})T(r,\mathcal{L})+O(\log r)\nonumber,
\eea
which  contradicts $k\geq2$.

{\bf(b)} Let $k=1$ and $gcd(n,m_1)=b_1$, therefore $gcd(n,n-m_1)=b_1$. Then from (\ref{000}) we have 
\bea\label{ec2} \frac{\mathcal{L}^n}{(\mathcal{L}-\beta_1)^{m_1}}=\frac{a^{'}(f-\beta_1)^{m_1}}{f^n}. \eea\\
If $z_0$ be a zero of $\mathcal{L}-\beta_1$ with
multiplicity $p$, then $z_0$ is a zero of $f$ with multiplicity $q$ such that $m_1p=nq$. Since $gcd(n,m_1)=b_1$, so $m_1=b_1 \eta$ and $n=b_1 \zeta$ for some $\eta,\zeta \in \mathbb{N}$, where $gcd(\eta,\zeta)=1$. Now $m_1 p = nq$ implies $b_1\eta p=q b_1\zeta$; i.e.,  $\eta p=q\zeta$. Since $gcd(\eta,\zeta)=1$, so $\zeta$ divides $p$. Therefore $p \geq \zeta$; i.e., $p\d\geq \frac{n}{b_1}$. Therefore, $\ol{N}(r,\beta_1;\mathcal{L})\leq \frac{b_1}{n}T(r,\mathcal{L})$ and similarly $\ol{N}(r,\beta_1;f)\leq \frac{b_1}{n}T(r,f)$ .\\
Also let $z_0$ be a pole of $f$ of multiplicity $\nu$, then $z_0$ will be a zero of $\mathcal{L}$ of multiplicity $\xi$ such that $\nu (n-m_1) = \xi n$. Since $gcd(n,n-m_1)=b_1$, hence $\frac{1}{\nu}\leq \frac{b_1}{n} $. Therefore $\ol{N}(r,\infty;f)\leq \frac{b_1}{n}T(r,f).$ Also we know that $\ol{N}(r,\infty;\mathcal{L})=O(\log r).$\\
Now from (\ref{ec2}) it is clear that $\ol{N}(r,0;\mathcal{
L})= \ol{N}(r,\infty;f)+\ol{N}(r,\beta_1;f).$ \\
Now using the second fundamental theorem in view of (\ref{0000}) we get \bea T(r,\mathcal{L})&\leq& \ol{N}(r,\beta_1;\mathcal{L})+\ol{N}(r,\infty;\mathcal{L})+\ol{N}(r,0;\mathcal{L})+S(r,\mathcal{L}) \nonumber\\ & 
\leq& \frac{3b_1}{n}T(r,\mathcal{L})+O(\log r)\nonumber
\eea
Which is a contradiction as $n\geq 3b_1+1$.
\end{proof}

\section{Proof Of The Theorems}

{\bf Proof of  Theorem 1.1 :} 
We  prove the theorem step by step as follows.

\underline{$\bf{(i) \implies (ii)}$ :
} Suppose $f$ be a non-constant meromorphic function and $\mathcal{L}$ be a non-constant L-function such that  
$E_f(S,t)=E_{\mathcal{L}}(S,t)$, where $t\in\mathbb{N}\cup\{0\}$. Consider $\mathcal
{F}$ and $\mathcal
{G}$ as defined by (\ref{el04}). Then for \begin{itemize}
\item [(a)]$t\geq 2$ and $n\geq 2s+5$, or
\item [(b)] $t=1$ and $n\geq 2s+6$, or
\item [(c)]$t=0$ and $n\geq 2s+11$,
\end{itemize} in view of the {\em Lemma \ref{l11}} we get 
$\d\frac{1}{\mathcal
{F}-1}=\frac{A}{\mathcal
{G}-1}+B$, where $A(\neq0), B\in \mathbb{C}$.
Hence we have \be\label{e1}T(r,\mathcal{F})=T(r,\mathcal{G})+O(1).\ee Since  \be\label{e3} T(r,\mathcal{F})=nT(r,f)+O(1)\;\; and\;\; T(r,\mathcal{G})=nT(r,\mathcal{L})+O(1).\ee So (\ref{e1}) implies that
\be\label{e2}T(r,f)=T(r,\mathcal{\mathcal{L}})+O(1).\ee
Now in view of $\ol{N}(r,\infty,\mathcal{L})=O(\log r)$ using (\ref{e3}) and (\ref{e2}) we get 
\beas &&
\ol{N}(r,0; \mathcal{F})+	\ol{N}(r,\infty; \mathcal{F})+	\ol{N}(r,0; \mathcal{G})+	\ol{N}(r,\infty; \mathcal{G})\\ &\leq& pT(r,f)+pT(r,\mathcal{L})+\ol N(r,\infty;f)+\ol N(r,\infty;\mathcal{L})\\
&=&(2p+1)T(r,f)+O(\log r)\\
&<&\frac{2p+2}{n}T(r,\mathcal{F})\\&\leq& T(r,\mathcal{F}) \;\;\; [\because n \geq 2p + 2].
\eeas
So in view of {\em Lemma \ref{l2}}, we have either $\mathcal{F}\mathcal{G} = 1$ or $\mathcal{F} = \mathcal{G}$.

Now again in view of {\em Lemma \ref{14}} we have $\mathcal{F} \mathcal{G}\neq 1$. Hence $\mathcal{F}= \mathcal{G}$. That is, we get 
\be\nonumber
P(f)=
P(\mathcal{L}),
\ee
which by condition (i) implies $f=\mathcal{L}$.\\

\underline{$\bf{(ii)\implies (i)}$ :}
Let $P(f)=P(\mathcal{L})$. That is, 
$$\prod_{i=1}^{p}(f-\alpha_i)^{m_i}=\prod_{i=1}^{p}(\mathcal{L}-\alpha_i)^{m_i},$$ which implies $f$ and $\mathcal{L}$ share $(S,\infty)$. Therefore, obviously $f$ and $\mathcal{L}$ share $(S,t)$ for $t\in\mathbb{N}\cup\{0\}$. Hence by condition (ii), we have $f= \mathcal{L}$. \\

{\bf Proof of Theorem 1.2 :} Let us consider $\mathbb{F}$ and $\mathbb{G}$ as defined by (\ref{el05}). 
Let $f$ be a non-constant meromorphic function and $\mathcal{L}$ be a non-constant L-function such that  
$E_f(S,t)=E_{\mathcal{L}}(S,t)$, where $t\in\mathbb{N}\cup\{0\}$. Then $\mathbb{F}$, $\mathbb{G}$ share $(1,t)$. 
Now for  \begin{itemize}
\item [(a)]$t\geq 2$ and $n\geq 2l+5$, or
\item [(b)] $t=1$ and $n\geq 2l+6$, or
\item [(c)] $t=0$ and $n\geq 2l+11$,
\end{itemize}
in view of  {\em Lemma \ref{l1.1}} we have
\bea\label{ep42}\frac{1}{\mathbb{F}-1}=\frac{A}{\mathbb{G}-1}+B,\eea where $A(\neq0), B\in\mathbb{C}$. 

From (\ref{ep42})  we easily obtain \bea \label{ep6}T(r,f) = T(r,\mathcal{L}) + S(r,f).\eea
\par Now in view of  $\ol{N}(r,\infty,\mathcal{L})=O(\log r)$, (\ref{ep6}) and from the construction of $\mathbb{F}$ and $\mathbb{G}$ we get 
\beas &&
\ol{N}(r,0; \mathbb{F})+	\ol{N}(r,\infty; \mathbb{F})+	\ol{N}(r,0; \mathbb{G})+	\ol{N}(r,\infty; \mathbb{G})
\\ &\leq&\ol{N}(r,0;f)+\sum_{j=1}^{k}\ol{N}(r,\beta_j;f)+\ol{N}(r,\infty;f)+\ol{N}(r,0;\mathcal{L})+\sum_{j=1}^{k}\ol{N}(r,\beta_j;\mathcal{L})+\ol{N}(r,\infty;\mathcal{L})
\\ &\leq& (2+k)T(r,f)+(1+k)T(r,\mathcal{L})+O(\log r)\\
&=&(2k+3)T(r,f)+O(\log r)\\
&<&\frac{2k+4}{n}T(r,\mathbb{F})\\\nonumber
&\leq&T(r,\mathbb{F}){\;\;}[\because\;\; n\geq 2k+4].
\eeas
So in view of {\em Lemma \ref{l2}}, we have either $\mathbb{F}\mathbb{G} \equiv 1$ or $\mathbb{F} \equiv \mathbb{G}$.
Again in view of {\em Lemma \ref{l04}} we have  $\mathbb{F}\mathbb{G} \not\equiv1$. Thus $\mathbb{F} \equiv \mathbb{G}$; i.e., $R(f)=R(\mathcal{L})$.

Therefore we find that $E_f(S, t)=E_\mathcal{L}(S,t)\implies f=\mathcal{L}$, whenever $R(f)=R(\mathcal{L})\implies f=\mathcal{L}$. That is $(i)\implies(ii).$ 
\vspace{0.1in}\par To show $(ii) \implies (i)$, suppose that $E_f(S, t)=E_\mathcal{L}(S,t)\implies f=\mathcal{L}$.
Let $R(f)=R(\mathcal{L})$, then we have $R(f)-1=R(\mathcal{L})-1$; i.e., $\d\frac{P(f)}{\phi(f)}=\d\frac{P(\mathcal{L})}{\phi(\mathcal{L})}.$ Therefore  $f$ and  $\mathcal{L}$ share $(S,\infty)$, which implies $E_f(S, t)=E_\mathcal{L}(S,t)$ and hence  $f=\mathcal{L}$.
\\\\{\bf Proof of Corollory \ref{c1} and Corollary \ref{c2}:}	The proof of {\em Corollory \ref{c1}} and {\em Corollary \ref{c2}} directly follows from the proof of {\em Theorem \ref{t1}} and {\em Theorem \ref{t2}} respectively, considering $\ol N(r,\infty;f)=O(\log r)$. Hence we omit the proof.
\begin{ack}
\par Sanjay Mallick is thankful to``Science and Engineering Research Board, Department of Science and Technology, Government of India"  for financial support to pursue this research work under the Project File No. EEQ/2021/000316.
\par	 Ripan Saha is thankful to the Department of Atomic Energy (DAE), India, for financial support to pursue this work $[$No. $0203/13(41)/2021$-$R\&D$-$II/13168]$. 	
\end{ack}

\end{document}